\documentclass[11pt, a4paper%, dvips
]{amsart}
\usepackage[%dvipdfm,
backref=page,bookmarksopen=true]{hyperref}
\usepackage[latin1]{inputenc}  % accents 8 bits dans le source
\usepackage[T1]{fontenc}       % accents dans le DVI
\usepackage[english]{babel}
\usepackage{amssymb}
\usepackage{latexsym}
\usepackage{mathrsfs}
\usepackage{color}
\usepackage{xspace}
\usepackage{graphics, graphicx, epsfig}
\usepackage[all, cmtip]{xy}

\newtheorem{thm}{Theorem}
\newtheorem*{thm*}{Theorem}

\newtheorem*{addendum*}{Addendum}

\newtheorem{lem}[thm]{Lemma}

\newtheorem*{convention*}{Convention}
\theoremstyle{definition}
\newtheorem*{defn*}{Definition}

\newtheorem*{scholium*}{Scholium}
\theoremstyle{remark}

\newtheorem*{example*}{Example}
\numberwithin{equation}{section}

\newcommand{\vareps}{\varepsilon}

\newcommand{\AAA}{\mathbb{A}}

\newcommand{\FF}{\mathbf{F}}

\newcommand{\ZZ}{\mathbf{Z}}

\newcommand{\SL}{\mathrm{SL}}

\newcommand{\inv}{^{-1}}

\newcommand{\cat}{{\upshape CAT(0)}\xspace}  %
\newcommand{\tangle}[2]% angle de Tits
{\angle_\mathrm{T}(#1,#2)}
\newcommand{\aangle}[3]% angle d'Alexandrov
{\angle_{#1}(#2,#3)}
\newcommand{\cangle}[3]% angle de comparaison
{\overline{\angle}_{#1}(#2,#3)}
\DeclareMathOperator{\dist}{dist}

\DeclareMathOperator{\Div}{Div}
\def\Aut{\mathop{\mathrm{Aut}}\nolimits}

\def\min{\mathop{\mathrm{min}}\nolimits}

\def\max{\mathop{\mathrm{max}}\nolimits}
\def\Op{\mathop{\mathrm{Opp}}\nolimits}
\title{Twin building lattices do not have asymptotic cut-points}

\author[P.-E. Caprace]{Pierre-Emmanuel Caprace*}
\address{Universit\'e catholique de Louvain, 1348 Louvain-la-Neuve\\ Belgium}
\email{pierre-emmanuel.caprace@uclouvain.be}
\thanks{*Supported by the Fund for Scientific Research--F.N.R.S., Belgium}
\author[F. Dahmani]{François Dahmani**}
\address{Institut de Math\'ematiques de Toulouse,
 Universit\'e Paul Sabatier  Toulouse III, 31062 Toulouse cedex 9, France}
\email{francois.dahmani@math.univ-toulouse.fr}
\thanks{**Supported by ANR grant 06-JCJC-0099}

\author[V. Guirardel]{Vincent Guirardel**}
\address{Institut de Math\'ematiques de Toulouse,
 Universit\'e Paul Sabatier  Toulouse III, 31062 Toulouse cedex 9, France}    %Universit\'e de Toulouse III, France}
\email{vincent.guirardel@math.univ-toulouse.fr}

\date{September 2009}
\begin{document}

\begin{abstract}
We show that twin building lattices have linear divergence, which
implies that all asymptotic cones are without cut-points.
\end{abstract}

\maketitle
%\subsection*{Divergence and asymptotic cut-points}

Studying a class of finitely generated groups up to quasi-isometry
amounts roughly  to studying the large scale geometry of these
groups. In this context, it is often crucial to understand the
geometry of their asymptotic cones. The following
describes a very basic topological property of asymptotic cones of
twin building lattices.

\begin{thm}\label{thm:cut-point}
The asymptotic cones of a twin building lattice do not admit
cut-points.
\end{thm}

%\begin{cor}\label{cor:Floyd}
In particular, it follows that twin building lattices are one-ended and have trivial Floyd boundary. 

We emphasize that twin building lattices are \emph{irreducible non-uniform} lattices in the product of the automorphism groups of the associated pair of twin buildings. In particular, the above does not follow directly from the elementary fact that asymptotic cones of non-trivial products have no cut points. We do not know whether an arbitrary finitely generated lattice in a product of at least two non-compact locally compact groups has automatically no asymptotic cut-points; the proof we shall present uses some geometric features which are specific to twin buildings.

%\qed
%\end{cor}
Following Drutu--Mozes--Sapir \cite[Proposition~1.1]{DMS}, the
statement of Theorem~\ref{thm:cut-point} is equivalent to the
linearity of the growth rate of a function associated to the
underlying group, called the \textbf{divergence}. Roughly speaking,
the divergence measures the length of a shortest path joining two
points at distance~$n$ in the Cayley graph and avoiding a ball
centered at a third point (see \emph{loc.~cit.}; more details will also be given below). It is
conjectured in \emph{loc.~cit.}, and established in a number of
special cases, that irreducible lattices in higher rank semi-simple
Lie groups all have linear divergence. The corresponding statement
for twin building lattices is established in
Theorem~\ref{thm:divergence} below, relying on the fact that twin
building lattices are not distorted in their ambient locally compact
groups~\cite{CR}. In the affine case, it covers in particular
arithmetic groups of the form $\SL_n(\FF_q[t, t\inv])$. In fact, our proof does not use the fact that the groups under considerations are lattices: the arguments hold for arbitrary groups acting strongly transitively on twin buildings. In particular the above result applies to Kac--Moody groups over arbitrary fields; recall however that the latter groups are finitely generated if and only if the associated ground field is finite.

\medskip
Throughout this note, we let $(W,S)$ denote a Coxeter system with
$S$ finite and $\AAA$ denote the standard apartment of type $(W,S)$,
\emph{i.e.} the corresponding Coxeter complex. The notation of the
present paper is taken over from~\cite{CR}. The reader is referred
to \emph{loc.~cit.} and references therein for basic facts on twin
buildings and their lattices.

\subsection*{A remark on the asymptotic cones of irreducible buildings}

We start with a general observation on the asymptotic geometry of irreducible buildings. Assume thus that $(W,S)$ is irreducible. 

If it is of spherical type, then any building of type $(W,S)$ is bounded and any of its asymptotic cones is therefore reduced to a single point. 

If $(W,S)$ is of affine type, then any asymptotic cone of a building of type $(W, S)$ is itself a (non-discrete) affine building (see \cite{Kleiner-Leeb}); such a building admits cut points if and only if it is a tree or, equivalently, if and only if $W$ is infinite dihedral. 

Finally, if $(W,S)$ is non-spherical and non-affine (\emph{i.e.} if $W$ is not virtually Abelian), then \cite[Cor.4.7 and Th.~5.1]{CF} ensures that any building of type $(W,S)$ admits \emph{rank one geodesics} (also called \emph{Morse geodesics}). Each asymptotic cone of such a building thus possesses cut-points by \cite[Proposition~3.24]{DMS}.

In particular, we conclude that the only non-spherical irreducible buildings whose asymptotic cones have no cut-points are the Euclidean buildings of dimension~$\geq 2$.

Finally, we point out that if $\Gamma$ is a finitely generated (possibly non-uniform) lattice of a building $X$ and if $\Gamma$ possesses an element $\gamma$ which acts as a rank one isometry on $X$, then $\gamma$ is a Morse element with respect to the word metric of $\Gamma$ (see \cite[Lemma~3.25]{DMS}). In particular  $\Gamma$ has asymptotic cut-points. The assumption that the lattices appearing in Theorem~\ref{thm:cut-point} act on \emph{twin} buildings is thus essential.

\subsection*{Pencils of parallel walls in Coxeter complexes}

\begin{lem}\label{lem:Coxeter}
Assume that $(W,S)$ is non-spherical. Then there exists a constant
$C \geq 1$, depending only on $(W,S)$, such that two chambers of
$\AAA$ at distance at least $Cn$ apart are separated by at least $n$
pairwise parallel walls.
\end{lem}

\begin{proof}
The group $W$ contains a torsion-free normal subgroup of finite
index, say $W_0$, such that for any wall $H$ and any $w \in W_0$,
either $w.H = H$ or $w.H \cap H = \varnothing$ (see
\cite[Lemma~1]{DranJan}). Let $C$ be the number of $W_0$-orbits of
walls in $\AAA$. Note that $C$ is a finite number.

Let now $x,y$ by two chambers of $\AAA$ such that $d(x,y) \geq Cn$.
Then $x$ and $y$ are separated by $Cn$ walls. By the pigeonhole
principle, at least $n$ walls amongst these lie in the same
$W_0$-orbit. By the definition of $W_0$, these walls must be
pairwise parallel.
\end{proof}

\subsection*{Disjoining an apartment from a ball}

Before presenting the main geometric lemma needed for the proof of Theorem~\ref{thm:cut-point}, let us briefly fix the notational conventions we shall adopt in the rest of the paper. 

We let $X = X_+ \times X_-$ be a product of a pair of thick buildings of type $(W,S)$ admitting a twinning. The set of pair of opposite chambers is denoted by $\Op(X)\subset X_+\times X_-$ and $\Gamma < \Aut(X)$ is a group preserving the twinning (and is thus a discrete subgroup with respect to the topology of uniform convergence on bounded subsets) and acting transitively on $\Op(X)$. Typical examples of groups admitting such an action are provided by twin building lattices. 
Note that the twinning prevents $\Gamma$ from acting co-compactly on  $X_+\times X_-$.

Following \cite{CR}, we endow $\Gamma$ with its generating sets consisting of those elements mapping some fixed element $x \in \Op(X)$ to an adjacent one.  Recall that two pairs of chambers $(c_+,c_-), (d_+,d_-)$ are adjacent if there is some $s \in S$ such that $c_\vareps$ shares an $s$-panel with $d_\vareps$, for both $\vareps=\pm$.
In particular $\Gamma$ is quasi-isometric to $\Op(X)$ with its induced path metric. Moreover, by \cite{CR}, this path metric is undistorted in $X$. 

\smallskip
The unique twin apartment determined by an opposite pair $x \in
\Op(X)$ will be denoted by $\AAA(x)$. Its respective projections in $X_+$ and $X_-$ are denoted by  $\AAA_+(x)$ and $\AAA_-(x)$.

Although the following result refers only to the combinatorial
gallery distance between chambers, its proof appeals to the \cat
realization of buildings~\cite{Davis} and its good convexity
properties.

%{\blWe will need the %TAG

\begin{lem}\label{lem:disjoining}
Assume that $(W,S)$ is non-spherical. Then there exists a constant
$C'\geq 1$ such that the following holds.

Consider $x= (x_+, x_-)$ and $c = (c_+, c_-)$ in $\Op(X)$,   $r \geq 1$
and   $\vareps \in \{+,-\}$. Suppose that
$$d_{X_\vareps}(x_\vareps, c_\vareps) \geq C'r.$$ % + d_{X_\vareps}(z_\vareps, \AAA(x)_\vareps).$$
Then there exists $x' = (x'_+, x'_-) \in \Op(X)$ such that
$d_X(x,x') \leq C'$ and $\AAA_\vareps (x')$ does not meet the ball of radius
$r$ around $c_\vareps$ in $X_\vareps$.
\end{lem}

In the course of the proof, we will use the following consequence of the axioms defining twin buildings: if $(c_+,c_-) \in \Op(X)$, $\pi$ is a panel of $c_+$, then among all chambers of $X_+$ sharing $\pi$ with $c_+$,  at most one is not opposite to $c_-$.

\begin{proof}
Let $C$ be the constant provided by Lemma~\ref{lem:Coxeter}, let $D
>0$ be twice the circumradius of a chamber in the \cat realization
of $X_\vareps$ and let $E >0$ denote the minimal distance (in the
\cat metric) between two parallel walls of an apartment. Notice that
$E$ is indeed positive since the Weyl group $W$ has finitely many
orbits on walls, and the stabilizer of each wall acts cocompactly on
that wall.

Let also $\dist$ denote the \cat distance on the buildings $X_+$ and
$X_-$ and on the standard apartment $\AAA$. Since $W$ acts properly
cocompactly on $W$, it follows that $(\AAA, \dist)$ is
quasi-isometric to the gallery distance on $\AAA$. Upon enlarging
$D$ if necessary, we may and shall therefore assume that for all
chambers $x, y \in X_\vareps $ and any $p \in x$ and $q \in y$, we
have
$$\frac 1 D d_{X_\vareps}(x,y) - D \leq \dist(p,q) \leq D d_{X_\vareps}(x,y).$$

By the `Parallel Wall Theorem' (which follows from
\cite[Theorem~2.8]{BrinkHowlett}), there exists a constant $C'' \geq
0$ such that in any apartment, a chamber at distance~$\geq C''$ from
a given wall is separated from that wall by a parallel wall. The
desired constant $C'$ is defined by
$$C' = \max\{2C''+1,  \frac{CD} E + 2D^2\} .$$

Let now $x$ and $c$ be opposite pairs such that
$d_{X_\vareps}(x_\vareps, c_\vareps) \geq C' r$. We need to find an
opposite pair $x'$ adjacent to $x$ and such that $\AAA(x')$ avoids
the ball of radius $r$ around $c$ in $\Op(X)$. To this end, let
$p_0$ denote the circumcentre of the chamber $c_\vareps$ and let $p$
denote the \cat projection of $p_0$ to  $\AAA(x)_\vareps$. Let also
$y$ be a chamber of $\AAA_\vareps(x)$ containing $p$. If $d_{X_\vareps}(y,
c_\vareps) > D^2( r + 1)$, then any chamber $v \in \AAA(x)_\vareps$
satisfies
$$d_{X_\vareps}(c_\vareps, v) \geq \frac 1 D \dist(p_0, p) \geq \frac 1 {D^2} d_{X_\vareps}(y, c_\vareps) - 1 > r$$
and hence $\AAA(x)_\vareps$ does not meet the ball of radius $r$
around $c_\vareps$. Thus we may take $x' = x$ and we are done in
this case. We assume henceforth that $d_{X_\vareps}(y, c_\vareps)
\leq D^2( r + 1)$.

We have
$$d_{X_\vareps}(x_\vareps, y) \geq d_{X_\vareps}(x, c_\vareps) - d_{X_\vareps}(y, c_\vareps) \geq (C' -D^2) r - D^2  \geq  \frac {CD} E r.$$
By Lemma~\ref{lem:Coxeter}, this implies that $x_\vareps$ and $y$
are separated by at least $n$ pairwise parallel walls, where $n =
\inf \{n \in \ZZ \; | \; n \geq \frac D E r\}.$ Amongst these walls,
let $H$ be the one which is closest to $x_\vareps$. By the Parallel
Wall theorem, we may assume, upon adding supplementary walls to our
pencil of parallel walls, that $d_{X_\vareps}(x_\vareps, H) < C''$.
Let $\alpha$ be the half-apartment of $\AAA(x)_\vareps$ bounded by
$H$ and containing $x_\vareps$ but not $y$. Let $x'' = (x''_+,
x''_-) \in \AAA(x)$ be such that $H$ contains a panel of
$x''_\vareps$, that $x''_\vareps \not \in \alpha$ and that
$d_{X_\vareps}(x_\vareps, x''_\vareps) \leq C''$.  In particular we
have $d_X(x, x'') \leq 2C''$. Now pick any opposite pair $x' =
(x'_+, x'_-)$ adjacent to but distinct from $x''$ 
 and such that
$x'_{-\vareps} = x''_{-\vareps}$ (by the remark before the proof, such a pair exists since $X$ is thick). 
Then $d_X(x, x') \leq 2C'' +1 \leq
C'$ and the apartment $\AAA_\vareps(x')$ shares with $\AAA_\vareps(x'')= \AAA_\vareps(x)$
the half-apartment $\alpha$. It remains to verify that 
%$\AAA(x')$
%does not meet the ball of radius $r$ around $z$ in $\Op(x)$. In fact
%it suffices to check that 
  $\AAA(x')_\vareps$ does not meet the
ball of radius $r$ around $c_\vareps$ in $X_\vareps$.

To check this, we first claim that the \cat orthogonal projection of
$p_0$ to $\AAA(x')_\vareps$ belongs the the wall $H$. Let
$\alpha(x)$ (resp. $\alpha(x')$) denote the complement of $\alpha$
in $\AAA(x)_\vareps$ (resp. $\AAA(x')_\vareps$). Since the
projection of $p_0$ to $\AAA(x)_\vareps$ belongs to $\alpha(x)$, it
follows that the projection of $p_0$ to $\alpha$ belongs to the wall
$H$.

Let now $\AAA$ denote the apartment of $X_\vareps$ which is the
union of $\alpha(x)$ and $\alpha(x')$. Then the characterization of
the \cat orthogonal projection in terms of Alexandrov angle
\cite[II.2.4]{Bridson-Haefliger} shows that the projection of $p_0$
to $\AAA$ coincides with $p$. As before, we deduce that the
projection of $p_0$ to $\AAA(x')_\vareps$ belongs to the wall
$\partial \alpha$. Since $\AAA(x')_\vareps = \alpha \cup
\alpha(x')$, the desired claim follows.

The claim implies that
$$\dist(p_0, \AAA(x')) \geq \dist(p, \alpha) \geq En \geq Dr.$$
Since the ball of combinatorial radius $r$ centered at $c$ is
contained in the ball of \cat radius $Dr$ centered at $p_0$, the
desired result follows.
\end{proof}

\subsection*{Divergence}\label{sec:divergence}

Let $\delta \in (0, 1)$ and $\lambda \geq 0$. Following \cite{DMS},
we define the divergence of a pair of points $(a, b)$ in a geodesic
metric space relative to a point $c$ to be the length of a shortest
path from $a$ to $b$ avoiding a ball around $c$ of radius $\delta
\dist(c, \{a, b\})- \lambda$. If there is no such path, then
the divergence is understood to be infinity. The divergence of the
pair $(a,b)$ is the supremum of the divergences relative to all
points $c$. Finally, the \textbf{divergence function}
$\Div_\lambda(n ; \delta)$ is the maximum of all divergences of all
pairs $(a, b)$ with $\dist(a, b) \leq n$. Proposition~1.1 from
\emph{loc.~cit.} ensures that the divergence function $\Div_2(n;
\frac 1 2)$ of a finitely generated group is linear if and only if
no asymptotic cone of the group admits cut-points. Therefore,
Theorem~\ref{thm:cut-point} from the introduction is a consequence
of the following.

\begin{thm}\label{thm:divergence}
The divergence  $\Div_0(n ;\frac{1}{2})$ of a twin building lattice
is linear.
\end{thm}

As we shall see, the proof below does not use the fact that the
groups in question are lattices. In fact, the arguments do not even
use the property that the underlying buildings are locally compact.
Therefore, the above result also holds for arbitrary groups endowed
with a root group datum.

\begin{proof}
Since twin building lattices are undistorted  (see 
\cite[Theorem~1.1]{CR}), it suffices to prove the result for the divergence
function of the metric space $(\Op(X), d_X)$. Let $a = (a_+, a_-) ,
b= (b_+, b_-), c= (c_+, c_-) \in \Op(X)$. Set $n = d_X(a,b)$ and
$$r = \min \{d_X(a, c), d_X(b, c)\}.$$
We shall construct a path of length $\leq Ln$ joining $a$ to $b$ in
$\Op(X)$ and avoiding the ball of radius $\frac r 2 $ around $c$,
where $L$ is a universal constant depending only on the type of the
building $X$. This implies the desired statement.

Upon exchanging $a$ and $b$, we may assume that $r = d(a,c) \leq
d(b,c)$. Moreover, if $n < r$, then any minimal path from $a$ to $b$
avoids the ball of radius $\frac r 2$ around $c$. We assume
henceforth that $n \geq r$.

Recall from the definition that a path $\gamma$ in $\Op(X)$ consists
of a pair of paths $\gamma_+$ and $\gamma_-$ in $X_+$ and $X_-$
respectively. A sufficient condition for $\gamma$ to avoid the ball
of radius $r$ around $c = (c_+, c_-)$ in $\Op(X)$ is that $\gamma_+$
avoids the ball of radius $r$ around $c_+$, or else that $\gamma_-$
avoids the ball of radius $r$ around $c_-$. This is the key
observation for the construction of desired path. In fact, we shall
construct it by concatenating $6$ pieces according to the following
scheme:
$$a \longrightarrow a' \longrightarrow  a'' \longrightarrow x \longrightarrow b'' \longrightarrow  b' \longrightarrow b.$$
In the successive pieces, the positive and  the negative part of the
path will alternately  avoid the ball of radius $\frac r 2$ around
$c_+$ and $c_-$, thereby providing a path satisfying the required
condition.

The intermediate points are defined as follows. The point $a' =
(a'_+, a'_-)$ is chosen to be any closest point to $a$ with respect
to the property that $d_{X_+}(c_+, a'_+) > C' (r+1) $ (possibly $a'
= a$). For this we just use the fact that in $X_+$, one can extend a geodesic path from $c_+$ to $a_+$ into an infinite geodesic ($X_+$ is non-spherical). One can promote this path into a path in $\Op(X)$ by choosing a suitable path in $X_-$. Thus, moving from $a$ to $a'$ can be achieved in less than $C'(r +1)+1$ steps, which costs a length of at most $2C'(r+1) +2 $ in $X$. Moreover, such a path does not get any closer to $c$ than $a$.

Likewise   $b' =  (b'_+, b'_-)$ is chosen to be any closest
point to $b$ with respect to the property that $d_{X_-}(c_-, b'_-)
> C' (r+1) $. 
%Notice that $d_X(a, a')\leq \frac r 2+1$ (resp.
%$d_X(b, b')\leq \frac r 2+1$). Moreover any minimal path from $a$ to
%$a'$ (resp. from $b$ to $b'$) avoids the ball of radius $\frac r 2$
%around $c$.

The point $a''$ is defined by means of Lemma~\ref{lem:disjoining}, applied to $x=a'$.
It lies at distance at most ${C'}$ from $a'$ and  $\AAA_+(a'')$ 
avoids the ball of radius $ r$ around $c_+$. Thus a shortest path from $a'$ to $a''$ cannot enter  the ball of radius
$ r $ around $c$.

%Moreover
%Lemma~\ref{lem:disjoining} also provides a path of length at
%most $C'$ in $\Op(X)$ which avoids the ball of radius
%$\frac r 2$ around $c$.

The point $b''$  is defined similarly; it lies at distance at most $
 {C'} $ from $b'$ and  $\AAA_-(b'')$  avoids the ball of radius
$ r $ around $c_-$.

Using Lemma~2.1 from~\cite{CR}, it is then possible to construct a
path $\gamma' = (\gamma'_+, \gamma'_-)$ of length $\leq 2
d_{X_-}(a''_-, b''_-)$ in $\Op(X)$  joining $a''$ to a some point $x
= (x_+, b''_-) \in \Op(X)$ so that the positive part $\gamma'_+$
remains in the twin apartment $\AAA_+(a'')$. In particular it avoids the
ball of radius $ r $ around $c_+$.

Finally, another application of  \cite[Lemma~2.1]{CR} provides a
path $ \gamma'' = (\gamma''_+, \gamma''_-)$ of length $\leq 2
d_{X_+}(x_+, b''_+)$ joining $x$ to $b''$ so that the negative part
$\gamma''_-$ remains in the twin apartment $\AAA_-(b'')$. In particular
it avoids the ball of radius $r$ around $c_-$.

It remains to estimate the total length of the path we have
constructed piecewise. The initial and terminal sections from $a$ to
$a''$ and from $b''$ to $b$ contribute  a piece of length at most
$ 2\times (2C'(r+1) +2 + C')\leq (10C' +4)r$. In particular we deduce
$$d_X(a'', b'') \leq d_X(a'',a) +d_X(a,b) +d_X(b'',b) \leq (10C'+4)r + n \leq (10C'+5)n$$
since $r \leq n$. Therefore, the section from $a''$ to $b''$ via $x$
contributes a piece of length at most
$$\begin{array}{rcl}
2  d_{X_-}(a''_-, b''_-) + 2 d_{X_+}(x_+, b''_+) & \leq & 2 d_X(a'', b'') + 2d_X(x,a'') + 2 d_X(a'',b'')\\
& \leq & 6  d_X(a'', b'')\\
& \leq & 6(10C'+5)n.
\end{array}$$
Adding up all the different contributions, we come to a total length
of at most $(70C' +40)n$, which is linear in $n$ as desired.
\end{proof}

\providecommand{\bysame}{\leavevmode\hbox to3em{\hrulefill}\thinspace}
\providecommand{\MR}{\relax\ifhmode\unskip\space\fi MR }
\providecommand{\MRhref}[2]{%
  \href{http://www.ams.org/mathscinet-getitem?mr=#1}{#2}
}
\providecommand{\href}[2]{#2}

\end{document}